\documentclass[11pt, english]{article}
 \pdfoutput=1
\usepackage{babel,amsthm,amsmath,amsfonts,amssymb,amscd,stmaryrd,yhmath}
\usepackage{amssymb,xspace,amscd,txfonts,esvect,yfonts,calrsfs}
\usepackage{mathrsfs}
\usepackage[utf8]{inputenc} 
\usepackage{enumitem}
\DeclareMathAlphabet{\mathcalligra}{T1}{calligra}{m}{n}

\usepackage{graphicx}
\input xy
\xyoption{all}
\CompileMatrices
\usepackage{geometry}
\usepackage{longtable}

\theoremstyle{plain}
\newtheorem{thm}{Theorem}[section]

\newtheorem{lem}[thm]{Lemma}

\newtheorem{cor}[thm]{Corollary}

\newtheorem{prop}[thm]{Proposition}

\theoremstyle{definition}

\newtheorem{defn*}{Definizione}
\theoremstyle{remark}

\newtheorem{rem}{Remark}

\numberwithin{equation}{section}

\newcommand\gt{\mathfrak{g}}
\newcommand\hg{\mathfrak{h}}

\newcommand\Sp{\mathbf{Sp}}

\def\re{\mathrm{Re}\,}

\newcommand\R{\mathbb{R}}

\newcommand\SL{\mathbf{SL}}
\newcommand\SU{\mathbf{SU}}

\newcommand\GL{\mathbf{GL}}
\newcommand\SO{\mathbf{SO}}

\newcommand\C{\mathbb{C}}

\newcommand\Z{\mathbb{Z}}

\newcommand\zt{\mathfrak{z}}

\newcommand{\Spin}{\mathbf{Spin}}

\newcommand{\G}{\mathbf{G}_2}

\usepackage{cite}

\usepackage{cleveref}
\crefname{section}{§}{§§}
\Crefname{section}{§}{§§}

\usepackage{longtable}

\usepackage{lscape}
 
\usepackage[affil-it]{authblk}

\begin{document} 
\title{Invariant cocalibrated $\G$-structures on nilmanifolds}
\author{Leonardo Bagaglini}
\affil{Dipartimento di Matematica ed Informatica {``}Ulisse Dini{"}\\Università degli Studi di Firenze\\Viale Morgagni, 67/a - 50134 Firenze, Italy\\\textnormal{{leonardo.bagaglini@unifi.it}}}
\date{November 6, 2015}
\maketitle
\abstract{We classify nilmanifolds admitting invariant cocalibrated $\G$-structures.}
\tableofcontents 
\section{Introduction}\par
A seven dimensional connected, oriented, Riemannian manifold $M$ with Holonomy contained in $\G$ is characterized by existence of a $\G$-reduction $\mathcal{P}$ of its orthogonal frame bundle $\mathcal{F}$, on which the Cartan form of the Levi-Civita connection restricts; in other words by the vanishing of the intrinsic torsion $\tau$ of $\mathcal{P}$, \cite{Sal}.\\
Alternatively one can observe that $\G\subset\SO(7)$ is defined as the stabiliser, under the $\GL(7,\R)-$action, of a stable $3-$form $\varphi_0$, or equivalently (the component of) the stabiliser of a stable $4-$form $\phi_0$, both defining the underlying euclidean metric and related by $\phi_0=*\varphi_0$.\\ Consequently the principal bundle $\mathcal{P}$ is defined by a form $\varphi$, or $\phi$, satisfying $\xi_*\varphi=\varphi_0$ and $\xi_*\phi=\phi_0$ for a suitable frame $\xi$. Then $\tau$ is related to $\nabla\varphi$, or $\nabla\phi$, and its vanishing is equivalent to both $d\varphi=0$ and $d\phi=0$, \cite{FG}.\par
For general linear principal $\G-$bundles, $\tau$ equals $\tau_0+\tau_1+\tau_2+\tau_3$, where each $\tau_h$ takes values in a bundle modelled on an irreducible $\G-$module $\chi_h$, \cite{FG}. Conditions $d\varphi=0$ and $d\phi=0$ are equivalent to $\tau_0=\tau_1=\tau_3=0$ and $\tau_1=\tau_2=0$ respectively; if the first holds we say the structure \emph{calibrated} and, if the second holds, \emph{cocalibrated}.  For instance codimension one submanifolds of $\Spin(7)$-manifolds and closed seven dimensional spin manifolds have naturally cocalibrated $\G$-structures \cite{CN}.\\
It is well known that other combinations of torsions are admissible, see \cite{Fer}, \cite{Cab}, \cite{CMS}, \cite{Fei}.\par
Generally the calibrated condition is more restrictive than its counterpart, in fact compact examples of manifolds with calibrated structures are not easy to exhibit. In \cite{CF} Conti and Fern\'andez classified nilmanifolds with a (left) invariant calibrated $\G$-structure. Some homogeneous cocalibrated structure was studied  in \cite{Cab}, and by Freibert in \cite{Fre1} and \cite{Fre2}. In this paper we present the classification of nilmanifolds admitting cocalibrated structures.\\
Any nilmanifold equipped with an invariant metric is spin, thus, by Theorem 1.8 from \cite{CN}, it has a cocalibrated $\G$-structure. As a result we obtain that nilmanifolds defined by Lie algebras in Tables \ref{tab1}, \ref{tab2}, \ref{tab4} and \ref{tab5} admit a cocalibrated $\G$-structure, but not an invariant one.\par\medskip
A nilmanifold is a compact manifold on which a nilpotent Lie group $G$ acts transitively, namely homogeneous of type $G/N$ with $N$ a discrete subgroup\cite{Mal}.
The Lie group $G$ can define a nilmanifold if and only if its Lie algebra $\gt$ admits a basis $\left\{e_1,\dots,e_7\right\}$ with rational structure constants; in this case the De Rham cohomology of $G/N$ can be calculate using invariant differential forms and it is isomorphic to the Chevalley-Eilenberg cohomology of $\gt$, see \cite{Nom}.
Invariant forms on $G/N$ are uniquely determined by forms on $\gt$, hence we can restrict our attention to seven dimensional nilpotent  Lie algebras (classified in \cite{Gon}).\par\medskip
The paper is structured as follows.\\
In Section \cref{sec:sec2} we recall theory of stable forms in order to describe algebraic properties of $\G$.\\
In Section \cref{sec:sec3} we classify $\G$-structures on seven dimensional nilpotent Lie algebras with respect to particular fibrations over six dimensional nilpotent Lie algebras equipped with $\SU(3)$-structures, and produce some obstructions to their existence.\\
Finally in Sections \cref{sec:sec4} and \cref{sec:sec5} we classify all seven dimensional nilpotent Lie algebras admitting cocalibrated $\G$-structures.\\
Appendix \cref{sec:app} contains some explicit calculations omitted from the rest of the paper.
\section{Stable forms}\label{sec:sec2}
In this section we describe some algebraic properties of $\G$ in terms of stable forms. We start recalling few classical general results on $\GL$-orbits of forms ( \cite{KS} ) and next we focus on form spaces over six and seven dimensional real vector spaces ( \cite{Rei} ). For a more detailed treatment see \cite{CLSS}.\medskip\par
Consider the natural $\GL(V)$-action on $\Lambda^kV^*$.\\
A $k$-form $\rho$ on an $n-$dimensional real vector space $V$ is said to be \emph{stable} if its orbit is open in $\Lambda^k V^*$.
\begin{thm}
Suppose $k\leq [\frac{n}{2}]$. Then stable $k$-forms (equivalently $(n-k)$-forms) exist if and only if $k\leq 2$ or $k=3$ and $n\in\left\{6,7,8\right\}$.\par
Moreover, $k$ fixed, there are finite many open orbits.
\end{thm}
\begin{thm} Suppose $n$ even and $k\in\left\{2,n-2\right\}$ or $n\in\left\{6,7,8\right\}$ and $k\in\left\{3,n-3\right\}$. Then there exists an $\GL(V)$-equivariant map 
\begin{equation}\label{vol}
\varepsilon:\Lambda^kV^*\rightarrow \Lambda^nV^*,
\end{equation}
homogeneous of degree $(\frac{n}{k})$, such that $\mathrm{Ker}(\varepsilon)=\left\{\text{non stable forms}\right\}$.\par
Moreover if we consider a stable $k$-form $\rho$ and define the \emph{dual} form $\hat{\rho}\in\Lambda^{n-k}V^*$ as $$d\varepsilon_\rho(\alpha)=\hat{\rho}\wedge \alpha,\quad\forall \alpha\in\Lambda^kV^*,$$
then $\hat{\rho}$ is stable, the identity component of its stabiliser equals that one of $\rho$ and verifies 
$$\hat{\rho}\wedge\rho=\frac{n}{k}\varepsilon(\rho).$$
\end{thm}
\subsection{Stable forms in six dimensions}\label{sec:sixdim}
Consider $n=6$.\par
In $\Lambda^2V^*$ there exists only one open orbit. In particular we have the following
\begin{thm}
The set $$\Lambda_0(V)=\left\{\omega\in\Lambda^2V^*\;|\;\omega^3\neq 0\right\},$$
is the only open orbit in $\Lambda^2V^*$.
Thus, given $\omega\in\Lambda_0(V)$, its stabiliser is isomorphic to $\Sp(6,\R)$, its volume form $\varepsilon(\omega)$ can be chosen equal to $\frac{1}{3!}\omega^3$ and its dual form $\sigma$ equal to $\frac{1}{2}\omega^2.$\par
Moreover, there exists a suitable co-frame $\{f^1,\dots,f^6\}$ such that 
\begin{align*}
&\omega=f^{12}+f^{34}+f^{56},\\&\sigma=f^{1234}+f^{1256}+f^{3456},\\&\varepsilon(\omega)=f^{123456}.
\end{align*}
\end{thm}
The space $\Lambda^3V^*$ has, instead, two different open orbits. They can be distinguished as follows (see \cite{Hit}). Consider $\rho\in\Lambda^3V^*$ and define\footnote{Let $\iota_x$ be the contraction by the vector $x$.}
$$k_\rho:V\ni x\mapsto \iota_x\rho\wedge\rho\in\Lambda^5V^*.$$
Thanks to the isomorphism $\theta^{-1}: V\otimes\Lambda^6V^*\ni x\otimes\alpha\mapsto \iota_x\alpha\in\Lambda^5V^*$ we can define 
\begin{gather*}
K_\rho:V\xrightarrow{k_\rho}\Lambda^5V^*\xrightarrow{\theta}V\otimes\Lambda^6V^*,\\\text{and}\\
\lambdaup(\rho)=\frac{1}{6}\mathrm{trace}({K_{\rho}}^2)\in\left({\Lambda^6V^*}\right)^{\otimes 2}.
\end{gather*}
\begin{thm}
Consider the sets\footnote{We say $\alpha\in\left({\Lambda^6V^*}\right)^{\otimes 2}$ positive (resp. negative) and write $\alpha>0$ (resp. $\alpha<0$) if $\alpha=\varepsilon\otimes\varepsilon$ (resp. $\alpha=-\varepsilon\otimes\varepsilon$).} $$\Lambda_+(V)=\left\{\rho\in\Lambda^3 V^*\,|\,\lambdaup(\rho)>0\right\}\quad \text{and}\quad\Lambda_-(V)=\left\{\rho\in\Lambda^3 V^*\,|\,\lambdaup(\rho)<0\right\}.$$ They are the only open orbits in $\Lambda^3V^*$.\par
If $\psi\in\Lambda_+(V)$ the identity component of its stabiliser is isomorphic to $\SL(3,\R)\times\SL(3,\R)$ and in a suitable co-frame $\left\{f^1,\dots,f^6\right\}$ results
$$\psi=f^{123}+f^{456}.$$\par
If $\psi\in\Lambda_-(V)$ the identity component of its stabiliser is isomorphic to $\SL(3,\C)$ and in a suitable co-frame $\left\{f^1,\dots,f^6\right\}$ results $$\psi=-f^{246}+f^{136}+f^{145}+f^{235}.$$\par
Moreover, if we consider a stable $3$-form $\psi$, fix a volume form $\varepsilon(\psi)\in\Lambda^6V^*$ and define  $$J_\psi=\frac{1}{\sqrt{|\lambdaup(\psi)|}}K_\psi\in V^*\otimes V,$$ then
\begin{align*}
\psi\in\Lambda_+(V)&\Leftrightarrow J_\psi\;\;\text{is a para-complex structure},\\\psi\in\Lambda_-(V)&\Leftrightarrow J_\psi\;\;\text{is a complex structure},
\end{align*}
and in both cases the dual form is $J_{\psi}^*\psi$.
\end{thm}
For our purpose we are interested in particular pairs of stable forms, described by the following
\begin{thm}
Let $(\omega,\psi_-)\in\Lambda_0(V)\times\Lambda_-(V)$, $J_{\psi_-}$ defined by the choose of $\varepsilon(\omega)\in\Lambda^6V^*$, $h\in V^*\otimes V^*$ defined by $$h(x,y)=\omega(x,J_{\psi_-}y),\quad\forall x,y\in V,$$ and suppose  $$\omega\wedge\psi_-=0.$$\\
If $h$ is positive definite then the stabiliser of the pair $(\omega,\psi_-)$ is a subgroup of $\SO(V,h)$ isomorphic to $\SU(3)$, \emph{i.e.} the pair defines an $\SU(3)$-structure where $J_{\psi_-}$ is the complex structure, $\Psi=-J_{\psi_-}^*\psi_- + i \psi_-$ the complex volume form and $h$ the underlying hermitian metric.\\ Further any other $\SU(3)$-structure is obtained in this way.\par
Moreover, considering the real part $\psi_+$ of $\Psi$, if $(\omega,\psi_-)$ is \emph{normalized}, $$\text{i.e.}\quad\psi_+\wedge \psi_-=\frac{2}{3}\omega^3,$$ then 
\begin{align*}
&\sigma=*_h \omega,\\
&\psi_+=*_h\psi_-,
\end{align*}
and there exists a suitable ($h$-orthonormal) co-frame $\left\{f^1,\dots,f^6\right\}$ such that 
\begin{align}\label{ortoh}
&\omega=f^{12}+f^{34}+f^{56},\nonumber\\
&\psi_-=-f^{246}+f^{136}+f^{145}+f^{235}.
\end{align}
\end{thm}
\subsection{Stable forms in seven dimensions}
Consider $n=7$.\par
Given a $3-$form $\varphi$ consider the symmetric $2$-form $b_\varphi$ on $V$ with values in $\Lambda^7V^*$, $$b_\varphi(x,y)=\iota_x\varphi\wedge\iota_y\varphi\wedge\varphi,\quad\forall x,y\in V.$$
Then the volume form map \eqref{vol} can be defined as $$\varepsilon(\varphi)=\sqrt[9]{\mathrm{det}(b_\varphi)}.$$
\begin{thm}
Let $\varphi$ be any stable $3$-form and $g_\varphi$ the symmetric $2$-form on $V$ defined by $$g_\varphi=\frac{1}{3\varepsilon(\varphi)}b_\varphi.$$
Then 
$$\Pi_+(V)=\left\{\varphi\in\Lambda^3 V^* \;|\;g_\varphi\; \text{is definite}\right\}\;\text{and}\;\Pi_-(V)=\left\{\varphi\in\Lambda^3 V^* \;|\;g_\varphi  \;\text{is indefinite}\right\},$$
are the only open orbits in $\Lambda^3 V^*$.\par
If $\varphi\in\Pi_+(V)$ we say it (and its dual form) \emph{positive}, its stabiliser is a subgroup of $\SO(V,g_\varphi)$ isomorphic to $\G$, its dual form $\phi$ equals $*_{g_\varphi}\varphi$ and in a suitable ($g_\varphi$-orthonormal) co-frame $\left\{f^1,\dots,f^7\right\}$ results 
\begin{align}\label{ortog}
&\varphi=f^{127}+f^{347}+f^{567}+f^{135}-f^{146}-f^{236}-f^{245},\nonumber\\
&\phi=f^{1234}+f^{1256}+f^{3456}+f^{2467}-f^{1367}-f^{1457}-f^{2357}.
\end{align}
\end{thm}
\begin{rem}
Since the stabiliser of a positive $4$-form is isomorphic to $\G\times\Z_2$, and the correspondence between positive forms $\left\{\varphi\mapsto\phi\right\}$ is $2:1$, a $\G$-structure is defined by $\phi$ together with an orientation (see \cite{Bry}).
\end{rem}
\section{Cocalibrated structures}\label{sec:sec3}
In this section we classify seven dimensional nilpotent Lie algebras with a cocalibrated $\G$-structure in terms of fibrations over six dimensional nilpotent Lie algebras endowed with a particular $\SU(3)$-structure.\medskip\par
Consider an oriented seven dimensional nilpotent real Lie algebra $\gt$ with center $\zt$ and volume form $\varepsilon\in\Lambda^7\gt^*$, a non zero vector $X\in\zt$ and the following short exact sequence of Lie algebras
\begin{equation}\label{fibration}
\begin{CD}
0@>>>\R X@>>>\gt@>\pi>>\hg@>>>0.
\end{CD}
\end{equation}
\begin{rem}
There is a natural isomorphism of real algebras
\begin{equation*}
\begin{CD}
\left\{\alpha\in\Lambda^\bullet\gt^*\,|\,\iota_X\alpha=0\right\}@>\pi_*>>\Lambda^\bullet\hg^*.\\
\end{CD}
\end{equation*}
\end{rem}
\begin{prop}\label{prop1}
Let $\phi$ be a stable $4$-form on $\gt$ defining a cocalibrated $\G$-structure compatible with the orientation, $g$ the underlying metric and $\eta$ the dual form of $\frac{1}{||X||}X$ with respect to $g$.\\ Then the pair of forms $(\omega,\psi_-)\in\Lambda^2\hg^*\times\Lambda^3\hg^*$,
\begin{align*}
&\psi_-=\pi_*\left(-\frac{1}{||X||}\iota_X\phi\right),\\
&\omega\;\;\text{s.t.}\;\;\begin{cases}\frac{1}{6}\omega^3=\pi_*(\frac{1}{||X||}\iota_X\varepsilon),\\\sigma=\frac{1}{2}\omega^2=\pi_*(\phi+\frac{1}{||X||}\iota_X\phi\wedge\eta),\end{cases}
\end{align*}
 defines an $\SU(3)$-structure on $\hg$, is normalized, and satisfies\footnote{Recall notation in \cref{sec:sixdim}.} 
\begin{align}
&d(\pi^*\psi_-)=0,\nonumber\\&\label{eq1}d(\pi^*\sigma)=(\pi^*\psi_-)\wedge d(\eta),\\
&\label{eq2}g=\pi^*h+\eta\cdot\eta,\\
&\label{eq3}\phi=\pi^*\sigma+(\pi^*\psi_-)\wedge\eta.
\end{align}\par
Vice versa let $(\omega,\psi_-)\in\Lambda^2\hg^*\times\Lambda^3\hg^*$ be a pair of stable forms defining an $\SU(3)$-structure on $\hg$, compatible with the orientation $\pi_*(\iota_X\varepsilon)$ and normalized; $\eta\in\Lambda^1\gt^*$ a $1-$form for which Equation \eqref{eq1} is satisfied and the symmetric $2-$form $\pi^*h+\eta\cdot\eta$ is positive definite.\\ Then the $4-$form $\phi$ defined by \eqref{eq3} induces a cocalibrated $\G$-structure with underlying metric $g$ defined by \eqref{eq2} and volume form $\varepsilon$.
\end{prop}
\begin{proof}
Define
\begin{align*}
&\psi_-=-\frac{1}{||X||}\iota_X\phi,\\
&\sigma=\phi-\psi_-\wedge\eta.
\end{align*}
It is easy to see\footnote{Compare normal forms \eqref{ortog} and \eqref{ortoh}.} that those are stable forms inducing an $\SU(3)$-structure on $X^{\perp}$\footnote{Here $X^\perp$ indicates the $g$-orthogonal complement of $\left\{X\right\}$ in $\gt$.}, and are normalized. Moreover, because of both $\psi_-$ and $\sigma$ vanish on $X$, $\hg$ inherits an $\SU(3)$-structure defined by the two stable forms $\pi_*\psi_-$ and $\pi_*\sigma$, which we will denote as $\psi_-$ and $\sigma$, and such that $\pi$ is an isometry on $X^\perp$.\par
From $L_X=0$ we deduce that forms $\psi_-$ and $\sigma$ satisfy Equation \eqref{eq1}, in fact
\begin{align*}
&d\psi_-=d\left(-\frac{1}{||X||}\iota_X\phi\right)=-\frac{1}{||X||}\iota_X(d(\phi))=0,\\&d\sigma=d(\phi-\psi_-\wedge\eta)=\psi_-\wedge d(\eta),
\end{align*}
and the metric $g$ satisfies Equation \eqref{eq2}.\par
For the converse suppose $\omega,\,\psi_-,\,\eta$ as in the statement. By definition $X^\perp$ admits an $\SU(3)$-structure defined by the pair $(\pi^*\omega,\pi^*\psi_-)$. Consider a $\pi^*h-$orthonormal frame $\left\{f_1,\dots,f_6\right\}$ of $X^\perp$ such that \eqref{ortoh} holds and put $f_7=\frac{1}{||X||}X$. Then the $4-$form $\phi$ defined by Equation \eqref{eq3} is stable and induces the metric $g$, in fact it has the normal form \eqref{ortog} with respect to the co-frame $\{f^1,\dots,f^7\}$. Furthermore Equation \eqref{eq1} implies $d(\phi)=0$.
\end{proof}
Thus we have the following obstruction to existence of cocalibrated $\G$-structures on nilpotent Lie algebras:
\begin{cor}\label{obs1}
Let $$Z^3_X=\left\{\pi_*(\iota_X\alpha)\in\Lambda^3\hg^*\,|\,\alpha\in\Lambda^4\gt^*,\;d(\alpha)=0\right\}.$$
If $$Z^3_X\cap\Lambda_-(\hg)=\emptyset$$
then $\gt$ does not admit a cocalibrated $\G$-structure.
\end{cor}
\begin{rem}
We will see that, except for the nilpotent Lie algebras listed in Tables \ref{tab2} and \ref{tab5},
when $\gt$ does not satisfy hypothesis of Corollary \ref{obs1}, then cocalibrated $\G$-structures exist.
\end{rem}
Another result which will be useful later is the following
\begin{lem}\label{lem2}
Let $(\omega,J)$ be an $\SU(3)$-structure on a six dimensional vector space $V$ with fundamental $2$-form $\omega$, orthogonal complex structure $J$ and hermitian metric $h$.\\
Then for any $J$-invariant $4-$dimensional subspace $W$ of $V$ results $\sigma|_W\neq 0$, where $\sigma$ is the dual form of $\omega$.
\end{lem}
\begin{proof}
Let $x,y\in W$ non zero vectors such that $\left\{x,Jx,y,Jy\right\}$ is a $h|_W$-orthonormal real basis of the space $W$, and $z$ a non zero unit vector in $W^\perp$. Then $\left\{x,Jx,y,Jy,z,Jz\right\}$ is a real $h$-orthonormal basis of $V$. It follows that
\begin{eqnarray*}
\sigma(x\wedge Jx\wedge y\wedge Jy)=(*_h\omega)(x\wedge Jx\wedge y\wedge Jy)=\omega(z\wedge Jz)=||z||^2=1,
\end{eqnarray*}
 thus $\sigma|_W\neq 0$. 
\end{proof}
\section{Decomposable case}\label{sec:sec4}
In this section we classify all seven dimensional decomposable nilpotent Lie algebras which admit a cocalibrated $\G$-structure.\medskip\par
Let $\gt$ be a seven dimensional decomposable Lie algebra and $\zt$ its center. 
Recall that an $\SU(3)$-structure with fundamental $2$-form $\omega$ and complex volume form $\Psi$ is said to be \emph{Half-Flat} if and only if $d\omega=0$ and $d\re(\Psi)=0$. 
\begin{lem}\label{lem1}
Suppose $\gt$ decomposes as $\hg\oplus\R$ where $\hg$ is a six dimensional nilpotent Lie algebra admitting an Half-Flat $\SU(3)$-structure and acting trivially on the factor $\R$ by the adjoint representation. Then $\gt$ admits a cocalibrated $\G$-structure.
\end{lem}
\begin{proof}
Let $$\gt=\hg\oplus\R e_7,\quad e_7\in\zt,$$
and $(\omega,\psi_+'+i\psi_-')$ an Half-Flat $\SU(3)$-structure on $\hg$, where $\omega$ is the fundamental $2$-form and $\psi_+',\psi_-'$ the real and imaginary part of the complex volume form. Observe that also $(\omega,-\psi_-'+i\psi_+')$ defines an $\SU(3)$-structure.
Let $\eta$ be the $1-$form on $\gt$ defined by 
\begin{align*}&\eta(e_7)=1,\\&\eta(Y)=0,\;\forall Y\in\hg.
\end{align*} By hypothesis it verifies $d \eta=0$. Therefore, putting 
\begin{align*}
& \omega=\omega',\\
&\psi_-=\psi_+',
\end{align*}
the triple $(\omega,\psi_-,\eta)$ satisfies conditions of Theorem \ref{prop1}, hence defines a cocalibrated $\G$-structure. 
\end{proof}
By the notation
\begin{gather*}
\gt=(de^1,\dots,de^7)=\left(\sum_{i_1j_1}c^{1}_{i_1j_1}{i_1j_1},\dots,\sum_{i_7j_7}c^{7}_{i_7j_7}{i_7j_7}\right),\\
\end{gather*}
we mean that $\left\{e^k\right\}$ is a basis of $\gt^*$ such that $de^k=\sum_{i<j}c^k_{ij}e^{ij},$ where $c_{ij}^k\in\R$ and $e^{ij}=e^i\wedge e^j$.\\
\begin{prop}\label{prop2}
Among all seven dimensional, decomposable, nilpotent, Lie algebras those admitting a cocalibrated $\G$-structure arise from the construction described in Lemma \ref{lem1}. Explicitly $$\gt=\hg\oplus\R,$$ where $\hg$ is a six dimensional, nilpotent, Lie algebra lying in Table \ref{tab3}.
\end{prop}
\begin{center}
\begin{longtable}{|l|}
\caption{Six dimensional nilpotent Lie algebras admitting Half-Flat $\SU(3)$-structures.}\label{tab3}\\
\hline
$\hg$\\
\endfirsthead
\multicolumn{1}{l}{\textit{Continues from the previous page}} \\\hline
\endhead
\hline
\multicolumn{1}{r}{\textit{Continues to the next page}}
\endfoot\hline
\endlastfoot\hline
$(0,0,0,0,0,0)$\\\hline
$2$-step nilpotent Lie algebras\\\hline
$(0,0,0,12,13,23)$\\
$(0,0,0,0,13-24,14+23)$\\
$(0,0,0,0,12,14+23)$\\
$(0,0,0,0,12,13)$\\
$(0,0,0,0,12,34)$\\
$(0,0,0,0,0,12+34)$\\
$(0,0,0,0,0,12)$\\\hline
$3$-step nilpotent Lie algebras\\\hline
$(0,0,0,0,12,15+34)$\\
$(0,0,0,0,12,14+25)$\\
$(0,0,0,12,23,14+35)$\\
$(0,0,0,12,23,14-35)$\\
$(0,0,0,12,13,14+35)$\\
$(0,0,0,12,13,14+23)$\\
$(0,0,0,12,13,24)$\\\hline
$4$-step nilpotent Lie algebras\\\hline
$(0,0,12,13,23,14)$\\
$(0,0,12,13,23,14+25)$\\
$(0,0,12,13,23,14-25)$\\
$(0,0,0,12,14,15+23)$\\
$(0,0,0,12,14-23,15+34)$\\
$(0,0,0,12,14,15)$\\
$(0,0,0,12,14,15+24)$\\
$(0,0,0,12,14,15+23+24)$\\\hline
$5$-step nilpotent Lie algebras\\\hline
$(0,0,12,13,14+23,24+15)$\\
\end{longtable}
\end{center}

\begin{proof}
Let $\gt$ be any decomposable seven dimensional nilpotent real Lie algebra.\par
If $\gt$ lies in Table \ref{tab1} then there exists a non zero vector $X\in\zt$ such that $$Z^3_X\cap\Lambda_-(\hg)=\emptyset,$$ hence, by Corollary \ref{obs1}, it does not admit a cocalibrated $\G$-structure.
\begin{center}
\begin{longtable}{|l|l|}\caption{Decomposable nilpotent Lie algebras satisfying $Z^3_X\cap\Lambda_-(\hg)=\emptyset$.}\label{tab1}\\
\hline
$\gt$&$X$\\\hline
\endfirsthead
\multicolumn{2}{l}{\textit{Continues from the previous page}} \\\hline
\endhead
\hline
\multicolumn{2}{r}{\textit{Continues to the next page}}
\endfoot\hline
\endlastfoot
\multicolumn{2}{|l|}{$3$-step nilpotent Lie algebras}\\\hline
$(0,0,0,0,23,34,36)$&$e_7$\\
$(0,0,0,0,12,15,0)$&$e_6$\\
$(0,0,0,12,13,14,0)$&$e_5$\\
$(0,0,0,12,14,24,0)$&$e_7$\\\hline
\multicolumn{2}{|l|}{$5$-step nilpotent Lie algebras}\\\hline
$(0,0,12,13,14,23+15,0)$&$e_6$\\$(0,0,12,13,14,15,0)$&$e_6$\\$(0,0,12,13,14,34-25,0)$&$e_7$\\
$(0,0,12,13,14+23,34-25,0)$&$e_7$\\
\end{longtable}
\end{center}\par
If $\gt$ appears in Table \ref{tab2}, let $\phi$ be any closed $4$-form. Put $X=e_7$, $\eta$ any $1-$form satisfying $\eta(X)>0$, and define, as in Proposition \ref{prop1}, forms $\sigma$ and $\psi_-$ on $\hg$.\\
By way of contradiction suppose $\phi$ stable. Then there exists a solution of Equations \eqref{eq1} and \eqref{eq2}. Now, for any volume form on $\hg$, the equation $$d(\psi_-)=0,$$ forces the subspace $W=\mathrm{Span}\left(\pi(e_3),\pi(e_4),\pi(e_5),\pi(e_6)\right)\subset\hg$ to be $J_{\psi_-}$-invariant. But one can check\footnote{Explicit calculations can be found in \cref{sec:app}.} that the equation $$d(\sigma)=\psi_-\wedge d(\eta)$$ implies $\sigma|_W=0$, contradicting Lemma \ref{lem2}.
\begin{center}
\begin{longtable}{|l|}
\caption{Decomposable nilpotent Lie algebras satisfying  $Z^3_X\cap\Lambda_-(\hg)\neq\emptyset$ for all $X\in\zt$ but admitting no cocalibrated $\G$-structures.}\label{tab2}\\\hline
$\gt$\\\hline
$3$-step nilpotent Lie algebras\\\hline
$(0,0,0,12,13-24,14+23,0)$\\$(0,0,0,12,14,13-24,0)$\\$(0,0,0,12,13+14,24,0)$\\\hline
\end{longtable}
\end{center}\par
Any other $\gt$ is the central extension of a six dimensional nilpotent Lie algebra $\hg$ appearing in Table \ref{tab3}. Such $\hg$ has an Half-Flat $\SU(3)$-structure (see \cite{Con}), hence, by Lemma \ref{lem1}, $\gt$ admits a cocalibrated $\G$-structure.
\end{proof}  
\section{Indecomposable case}\label{sec:sec5}
In this section we classify all seven dimensional indecomposable nilpotent Lie algebras which admit a cocalibrated $\G$-structure.
\medskip\par
Let $\gt$ be a seven dimensional decomposable Lie algebra and $\zt$ its center.
\begin{prop}\label{prop3}
Among all seven dimensional, indecomposable, nilpotent, Lie algebras those admitting a cocalibrated $\G$-structure are listed in tables \ref{tab2step} - \ref{tab5step}.
\end{prop}
\begin{longtable}{|l|l|}\caption{Indecomposable $2$-step nilpotent Lie algebras admitting cocalibrated $\G$-structures.}\label{tab2step}\\\hline
$\gt$&$(d e^1,\dots, d e^7)$\\\hline
\endfirsthead
\multicolumn{2}{l}{\textit{Continues from the previous page}} \\\hline
\endhead
\hline
\multicolumn{2}{r}{\textit{Continues to the next page}}
\endfoot\hline
\endlastfoot
$37A$ &$(0,0,0,0,12,23,24)$\\
$37B$ &$(0,0,0,0,12,23,34)$\\
$37C$ &$(0,0,0,0,12+34,23,24)$\\
$37D$ &$(0,0,0,0,12+34,13,24)$\\
$17$ &$(0,0,0,0,0,0,12+34+56)$\\
$37B_1$ &$(0,0,0,0,12-34,13+24,14)$\\
$37D_1$ &$(0,0,0,0,12-34,13+24,14-23)$\\
\end{longtable}
\begin{longtable}{|l|l|}\caption{Indecomposable $3$-step nilpotent Lie algebras admitting cocalibrated $\G$-structures.}\label{tab3step}\\\hline
$\gt$&$(d e^1,\dots, d e^7)$\\
\endfirsthead
\multicolumn{2}{l}{\textit{Continues from the previous page}} \\\hline
\endhead
\hline
\multicolumn{2}{r}{\textit{Continues to the next page}}
\endfoot\hline
\endlastfoot
\hline
$357A$ &$(0,0,12,0,13,24,14)$\\
$257A$  &$(0,0,12,0,0,13+24,15)$\\
$257C$ &$(0,0,12,0,0,13+24,25)$\\
$257I$ &$(0,0,12,0,0,13+14,15+23)$\\
$257J$  &$(0,0,12,0,0,13+24,15+23)$\\
$247A$ &$(0,0,0,12,13,14,15)$\\
$247B$ &$(0,0,0,12,13,14,35)$\\
$247C$ &$(0,0,0,12,13,14+35,15)$\\
$247D$ &$(0,0,0,12,13,14,25+34)$\\
$247F$ &$(0,0,0,12,13,24+35,25+34)$\\
$247I$ &$(0,0,0,12,13,25+34,35)$\\
$247J$ &$(0,0,0,12,13,15+35,25+34)$\\
$247L$ &$(0,0,0,12,13,14+23,15)$\\
$247M$  &$(0,0,0,12,13,14+23,35)$\\
$247N$ &$(0,0,0,12,13,15+24,23)$\\
$247O$ &$(0,0,0,12,13,14+35,15+23)$\\
$247P$ &$(0,0,0,12,13,23,25+34)$\\
$247Q$ &$(0,0,0,12,13,,14+23,25+34)$\\
$157$ &$(0,0,12,0,0,0,13+24+56)$\\
$147A$ &$(0,0,0,12,13,0,16+25+34)$\\
$147B$ &$(0,0,0,12,13,0,14+26+35)$\\
$147D$ &$(0,0,0,12,23,-13,15+26+16-2*34)$\\
$137A$ &$(0,0,0,0,12,34,15+36)$\\
$137B$ &$(0,0,0,0,12,34,15+36+24)$\\
$137C$ &$(0,0,0,0,12,14+23,16-35)$\\
$257J_1$  &$(0,0,12,0,0,13+14+25,15+23)$\\
$247F_1$ &$(0,0,0,12,13,24-35,25+34)$\\
$247P_1$ &$(0,0,0,12,13,23,24+35)$\\
$147A_1$  &$(0,0,0,12,13,0,16+24+35)$\\
$137A_1$ &$(0,0,0,0,13+24,14-23,15+26)$\\
$137B_1$ &$(0,0,0,0,13+24,14-23,15+26+24)$\\
$147E(\muup\neq 0,1)$&$(0,0,0,,12,23,-13,\muup 26-15-(-1+\muup)34)$\\
$147E_1(\muup>1)$&$(0,0,0,12,23,-13,2[26-34-\frac{\muup}{2}16+\frac{\muup}{2}25])$\\
\end{longtable}
\begin{longtable}{|l|l|}\caption{Indecomposable $4$-step nilpotent Lie algebras admitting cocalibrated $\G$-structures.}\label{tab4step}\\\hline
$\gt$&$(d e^1,\dots, d e^7)$\\
\endfirsthead
\multicolumn{2}{l}{\textit{Continues from the previous page}} \\\hline
\endhead
\hline
\multicolumn{2}{r}{\textit{Continues to the next page}}
\endfoot\hline
\endlastfoot
\hline
$2457A$ &$(0,0,12,13,0,14,15)$\\
$2457B$  &$(0,0,12,13,0,25,14)$\\
$2457C$ &$(0,0,12,13,0,14+25,15)$\\
$2457D$ &$(0,0,12,13,0,14+23+25,15)$\\
$2457E$  &$(0,0,12,13,0,23+25,14)$\\
$2457F$ &$(0,0,12,13,0,14+23,15)$\\
$2457G$ &$(0,0,12,13,0,15+23,14)$\\
$2457H$ &$(0,0,12,13,0,23,14+25)$\\
$2457I$ &$(0,0,12,13,0,14+23,25)$\\
$2457J$ &$(0,0,12,13,0,14+23,23+25)$\\
$2457K$ &$(0,0,12,13,0,15+23,14+25)$\\
$2457L$ &$(0,0,12,13,23,14+25,15+24)$\\
$2457M$ &$(0,0,12,13,23,24+15,14)$\\
$2357A$  &$(0,0,0,12,14+23,23,15-34)$\\
$2357B$ &$(0,0,0,12,14+23,13,15-34)$\\
$2357C$ &$(0,0,0,12,14+23,24,15-34)$\\
$2357D$ &$(0,0,12,14+23,13+24,15-34)$\\
$1357A$ &$(0,0,0,12,14+23,0,15+26-34)$\\
$1357B$ &$(0,0,0,12,14+23,0,15+36-34)$\\
$1357C$ &$(0,0,0,12,14+23,0,15+24+36-34)$\\
$1357D$ &$(0,0,12,0,23,24,16+25+34)$\\
$1357F$ &$(0,0,12,0,23,24,13+25-46)$\\
$1357G$ &$(0,0,12,0,23,14,16+25)$\\
$1357H$ &$(0,0,12,0,23,14,16+25+26-34)$\\
$1357I$ &$(0,0,12,0,23,14,25+46)$\\
$1357J$  &$(0,0,12,0,23,14,13+25+46)$\\
$1357L$ &$(0,0,12,0,13+24,23,16+25)$\\
$1357O$ &$(0,0,12,0,13+24,23,15+26+34)$\\
$1357P$  &$(0,0,12,0,13+24,23,15+26+34)$\\
$1357Q$ &$(0,0,12,0,13,23+24,15+26)$\\
$1357R$ &$(0,0,12,0,13,23+24,16+25+34)$\\
$2457L_1$ &$(0,0,12,13,23,14-25,15+24)$\\
$2357D_1$  &$(0,0,0,12,14+23,13-24,15-34)$\\
$1357F_1$ &$(0,0,12,0,23,24,13+25+46)$\\
$1357P_1$ &$(0,0,12,0,13+24,23,15+34-26)$\\
$1357Q_1$  &$(0,0,12,0,13,23+24,15-26)$\\
$1357M(\muup\neq 0,-1)$&$0,0,12,0,24+13,14,-(-1+\muup)34+15+\muup 26)$\\
$1357N(\muup\neq -2)$&$(0,0,12,0,13+24,14,46+34+15+\muup 23)$\\
$1357S(\muup\neq 1)$&$(0,0,12,13,24+23,25+34+16+15+\muup 26)$\\
$1357QRS_1(\muup\neq 0)$&$(0,0,12,0,13+24,14-23,\muup 26+15-(-1+\muup)34)$\\
\end{longtable}
\begin{longtable}{|l|l|}\caption{Indecomposable $5$-step nilpotent Lie algebras admitting cocalibrated $\G$-structures.}\label{tab5step}\\\hline
$\gt$&$(d e^1,\dots, d e^7)$\\
\endfirsthead
\multicolumn{2}{l}{\textit{Continues from the previous page}} \\\hline
\endhead
\hline
\multicolumn{2}{r}{\textit{Continues to the next page}}
\endfoot\hline
\endlastfoot
\hline
$23457C$ &$(0,0,12,13,14,15,25-34)$\\
$23457D$  &$(0,0,12,13,14,15+23,25-34)$\\
$23457E$ &$(0,0,12,13,14+23,15+24,23)$\\
$23457G$ &$(0,0,12,13,14+23,15+24,25-34)$\\
$13457D$  &$(0,0,12,13,14+23,0,15+24+26)$\\
$13457F$ &$(0,0,12,13,14,23,15+26)$\\
$12457A$ &$(0,0,12,13,0,14+25,16+35)$\\
$12457B$ &$(0,0,12,13,0,14+25,16+25+35)$\\
$12457C$ &$(0,0,12,13,0,14+25,26-34)$\\
$12457D$ &$(0,0,12,13,0,14+25,15+26-34)$\\
$12457E$ &$(0,0,12,13,0,14+23+25,16+24+35)$\\
$12457F$ &$(0,0,12,13,0,14+23+25,26-34)$\\
$12457G$ &$(0,0,12,13,0,14+23+25,15+26-34)$\\
$12457H$  &$(0,0,12,13,23,15+24,16+34)$\\
$12457I$ &$(0,0,12,13,23,15+24,16+34)$\\
$12457J$ &$(0,0,12,13,23,15+24,16+14+25+34)$\\
$12457K$ &$(0,0,12,13,23,15+24,16+14+34)$\\
$12457L$ &$(0,0,12,13,23,15+24,16+26+34-35)$\\
$12357A$ &$(0,0,0,12,14+23,15-34,16-35)$\\
$12357B$ &$(0,0,0,12,14+23,15-34,16+23-35)$\\
$12357C$ &$(0,0,0,12,14+23,15-34,16+24-35)$\\
$12457J_1$ &$(0,0,12,13,23,24+15,16+14-25+34)$\\
$12457L_1$ &$(0,0,12,13,23,-14-25,16-35)$\\
$12457N_1$ &$(0,0,12,13,23,-14-25,16-35+25)$\\
$12357B_1$ &$(0,0,0,12,14+23,15-34,16-23-35)$\\
$12457N(\muup\in\R)$&$(0,0,12,13,23,24+15,\muup 25 +26+34-35+16+14)$\\
$123457I(\muup\in\R)$&$(0,0,12,13,14+23,\muup 25+26+34-35+16+14)$\\
$12457N_2(\muup\geq 0)$&$0,0,12,13,23,-14-25,15+16+24-35+\muup 25)$\\
\end{longtable}
\begin{proof}
Let $\gt$ be any indecomposable seven dimensional nilpotent real Lie algebra.\par
If $\gt$ appears in Table \ref{tab4} then there exists a non zero vector $X\in\zt$ satisfying $$Z^3_X\cap\Lambda_-(\hg)=\emptyset,$$ hence, by Corollary \ref{obs1}, admits no cocalibrated $\G$-structures.
\begin{center}
\begin{longtable}{|l|l|l|}\caption{Indecomposable nilpotent Lie algebras satisfying $Z^3_X\cap\Lambda_-(\hg)=\emptyset$.}\label{tab4}\\
\endfirsthead
\multicolumn{2}{l}{\textit{Continues from the previous page}} \\\hline
\endhead
\hline
\multicolumn{2}{r}{\textit{Continues to the next page}}
\endfoot\hline
\endlastfoot
\hline
$\gt$&$(d e^1,\dots,d e^7)$&$X$\\\hline
\multicolumn{3}{|l|}{$2$-step nilpotent Lie algebras}\\\hline
$27A$&$(0,0,0,0,0,12,14+35)$&$e_6$\\
$27B$&$(0,0,0,0,0,12+34,15+23)$&$e_6$\\\hline
\multicolumn{3}{|l|}{$3$-step nilpotent Lie algebras}\\\hline
$257B$&$(0,0,12,0,0,13,14+25)$&$e_7$\\
$257D$&$(0,0,12,0,0,13+24,14+25)$&$e_7$\\
$257E$&$(0,0,12,0,0,13+45,24)$&$e_7$\\
$257G$&$(0,0,12,0,0,13+45,15+24)$&$e_7$\\
$257H$&$(0,0,12,0,0,13+24,45)$&$e_7$\\
$257K$&$(0,0,12,0,0,13,23+45)$&$e_7$\\
$257L$&$(0,0,12,0,0,13+24,23+45)$&$e_7$\\
$247E$&$(0,0,0,12,13,14+15,25+34)$&$e_7$\\
$247G$&$(0,0,0,12,13,14+15+24+35,25+34)$&$e_7$\\
$247H$&$(0,0,0,12,13,14+24+35,25+34)$&$e_7$\\
$247K$&$(0,0,0,12,13,14+35,25+34)$&$e_7$\\
$247R$&$(0,0,0,12,13,14+15+23,25+34)$&$e_7$\\
$247E_1$&$(0,0,0,12,13,14,24+35)$&$e_6$\\
$247H_1$&$(0,0,0,12,13,14+24-35,25+34)$&$e_7$\\\hline
\multicolumn{3}{|l|}{$4$-step nilpotent Lie algebras}\\\hline
$1457A$&$(0,0,12,13,0,0,14+56)$&$e_7$\\
$1457B$&$(0,0,12,13,0,0,23+14+56)$&$e_7$\\
$1457E$&$(0,0,12,0,23,24,25+46)$&$e_7$\\
$1357M(\muup=-1)$&$(0,0,12,0,24+13,14,-(-1+\muup)34+15+\muup 26),\;\muup=-1$&$e_7$\\
$1357N(\muup=-2)$&$(0,0,12,0,24+13,14,46+34+15+\muup 23),\;\muup=-2$&$e_7$\\\hline
\multicolumn{3}{|l|}{$5$-step nilpotent Lie algebras}\\\hline
$23457B$&$(0,0,12,13,14,25-34,23)$&$e_6$\\
$23457F$&$(0,0,12,13,14+23,25-34,23)$&$e_7$\\
$13457A$&$(0,0,12,13,14,0,15+26)$&$e_7$\\
$13457B$&$(0,0,12,13,14,0,15+23+26)$&$e_7$\\
$13457C$&$(0,0,12,13,14,0,16+25-34)$&$e_7$\\
$13457E$&$(0,0,12,13,14+23,0,16+25-34)$&$e_7$\\
$13457G$&$(0,0,12,13,14,23,16+24+25-34)$&$e_7$\\
$13457I$&$(0,0,12,13,14,23,15+25+26-34)$&$e_7$\\
\end{longtable}
\end{center}\par
If $\gt$ appears in Table \ref{tab5}, as we have seen in the proof of Proposition \ref{prop2} for Lie algebras in Table \ref{tab2}, putting $X=e_7$, $W=\mathrm{Span}\left(\pi(e_3),\pi(e_4),\pi(e_5),\pi(e_6)\right)\subset\hg$ and proceeding by contradiction one can prove that it admits no cocalibrated $\G$-structures.
\begin{center}
\begin{longtable}{|l|}\caption{Indecomposable nilpotent Lie algebras satisfying  $Z^3_X\cap\Lambda_-(\hg)\neq\emptyset$ for all $X\in\zt$ but admitting no cocalibrated $\G$-structures.}\label{tab5}\\\hline
$\gt$\\\hline
$3$-step nilpotent Lie algebras \\\hline
$(0,0,12,0,13,23,14)$\\$(0,0,12,0,13+24,23,14)$\\
\hline
$5$-step nilpotent Lie algebra \\\hline$(0,0,12,13,14,15,23)$\\\hline
\end{longtable}
\end{center}\par
For any other Lie algebra $\gt$ we can choose a non zero vector $X\in\zt$ and define a triple of forms $(\omega,\psi_-,\eta)$ (see Table \ref{tab6}) which satisfies hypothesis of Proposition \ref{prop1}, and consequently defines a cocalibrated $\G$-structure. Those algebras are listed in tables \ref{tab2step} - \ref{tab5step}.
\end{proof}
\section{Concluding remarks}
We have thus classified all seven dimensional nilpotent Lie algebras admitting cocalibrated $\G$-structures. We obtained several such structures according to the fact that, on a closed manifold, if there exists a $\G-$structure then there exists also a cocalibrated one, \cite{CN}. Our classification concerns the invariant case which turns out to be less sporadic, although rich enough to contain many samples.\\
Now, following \cite{FFM}, we are interested in studying nilsoliton metrics induced by cocalibrated structures, in order to find more restrictive existence conditions.     
Moreover we hope to analyse the behaviour of certain flows of $\G$-structures in such context: in particular the modified laplacian co-flow introduced in \cite{Gr}, for which short-time existence is proved.   
\section*{Acknowledgements}
I would like to thank Diego Conti, Marisa Fern\'andez and Anna Fino for useful improvements and punctual comments. Moreover I am very grateful to my advisor, Prof. Anna Fino, for her interest in my work.\\
This work was partially supported by the research group GNSAGA of the Istituto Nazionale di Alta Matematica Italiana ( INdAM ).
\section{Appendix}\label{sec:app}
In this appendix we perform some explicit calculations omitted from Propositions \ref{prop2} and \ref{prop3}.\medskip\par
Let $\gt$ be a nilpotent seven dimensional real Lie algebra, $\zt$ its center, $0\neq X\in\zt$ and $\hg$ defined by \eqref{fibration}.\\
Fix a basis $\left\{e_1,\dots,e_7\right\}$ of $\gt$ and consider an arbitrary closed form $\phi\in\Lambda^4\gt^*$,
\begin{equation*}
\phi=\sum_{i<j<k<l} \phi_{ijkl}e^{ijkl},\quad \phi_{ijkl}\in\R.
\end{equation*}\medskip
\subsection*{$\gt=(0,0,0,0,12,15,0)$}
Let $X=e_6$, $\hg=(df^1,\dots,df^6)=(0,0,0,0,12,0)$, where 
\begin{align*}
&f_{i}=\pi(e_i),\quad i<6,\\
&f_6=\pi(e_7).
\end{align*}
and consider
$$\psi_-=\pi_*(-\iota_{e_6}\phi)=\sum_{i<j<k}\phi_{ij67}f^{ij6}-\phi_{ijk6}f^{ijk}.$$\par
It turns out that
\begin{equation*}
d(\phi)=0 \quad\Longleftrightarrow\quad \begin{cases}
\phi_{2346}=0,\\
\phi_{2367}=0,\\
\phi_{2467}=0,\\
\phi_{3456}=0,\\
\phi_{3457}=0,\\
\phi_{3467}=0,\\
\phi_{3567}=0,\\
\phi_{4567}=0.\\
\end{cases}
\end{equation*}
Then it results
$$\lambdaup(\psi_-)=(\phi_{1346}\phi_{2567}+\phi_{1367}\phi_{2467}-\phi_{1467}\phi_{2356})^2 f^{123456}\otimes f^{123456},$$
hence hypothesis of Corollary \ref{obs1} are satisfied. It follows that $\gt$ admits no cocalibrated $\G$-structures.\medskip
\subsection*{$\gt=(0,0,0,12,13-24,23+14,0)$}
Let $X=e_7$ and $\hg=(df^1,\dots,df^6)=(0,0,0,12,13-24,23+14)$, where 
\begin{align*}
&f_{i}=\pi(e_i),\quad i<7.
\end{align*}\par
It turns out that
\begin{equation*}
d(\phi)=0 \quad\Longleftrightarrow\quad \begin{cases}
\phi_{1567}=0,\\
\phi_{2356}=-\phi_{1456},\\
\phi_{2456}=\phi_{1356},\\
\phi_{2457}=\phi_{2367}+\phi_{1467}+\phi_{1357},\\
\phi_{2567}=0,\\
\phi_{3456}=0,\\
\phi_{3457}=0,\\
\phi_{3467}=0,\\
\phi_{3567}=0,\\
\phi_{4567}=0.
\end{cases}
\end{equation*}
By way of contradiction suppose $\phi$ to be stable and consider
\begin{align*}
&\psi_-=\pi_*(-\iota_{e_7}\phi)=\sum_{i<j<k}\phi_{ijk7}f^{ijk},\\
&\eta=\sum_i c_ie^i,\quad c_i\in\R,
\end{align*}
where $\eta$ is the dual form of $e_7$ with respect the metric induced by $\phi$.\\
Then $\psi_-$ lies in $\Lambda_-(\hg)$ and for any volume form $\varepsilon\in\Lambda^6\hg^*$ results
$$J_{\psi_-}=\left(\begin{matrix}
*&*&0&0&0&0\\
*&*&0&0&0&0\\
*&*&*&*&0&0\\
*&*&*&*&0&0\\
*&*&*&*&*&*\\
*&*&*&*&*&*\\
\end{matrix}\right)$$
with respect to the basis $\left\{f_1,\dots,f_6\right\}$, hence the $4$-dimensional subspace $W=\mathrm{Span}(f_3,\dots, f_6)$ is $J_{\psi_-}$-invariant.\par
By assumption the $4-$form $$\sigma=\pi_*(\phi-\pi^*\psi_-\wedge\eta)=\sum_{i<j<k<l}\sigma_{ijkl}f^{ijkl},$$
is the Hodge dual of the fundamental $2$-form $\omega$ of an $\SU(3)$-structure on $\hg$, but it is easy to see that $$\sigma(f_3,f_4,f_5,f_6)=\sigma_{3456}=\phi_{3456}=0,$$  
 or equivalently $\sigma|_W=0$, contradicting Lemma \ref{lem2}. Hence $\gt$ admits no cocalibrated $\G$-structures.\par

\footnotesize{
\begin{longtable}{|l|l|l|}\caption{Explicit cocalibrated $\G$-structures $\phi=\frac{1}{2}\pi^*\omega^2+(\pi^*\psi_-)\wedge \eta$ on indecomposable nilpotent Lie algebras.\\ Coefficients $C_1,C_2$ are chosen so that $d(\frac{1}{2}\omega^2)=\psi_-\wedge d(\eta)$ holds, $B$ so that $\psi_+\wedge\psi_-=\frac{2}{3}\omega^3$ holds, and $A$ so that $\mathrm{det}(g)>0$  holds.}\label{tab6}\\\hline
&$\pi^*\omega$&\\\cline{2-2}
$\gt$&$\pi^*\psi_-$&Restrictions\\\cline{2-2}
&$\eta$&\\\hline
\endfirsthead
\multicolumn{3}{l}{\textit{Continues from the previous page}} \\\hline
\endhead
\hline
\multicolumn{3}{r}{\textit{Continues to the next page}}
\endfoot
\hline
\endlastfoot
\multicolumn{3}{|l|}{$2$-step nilpotent Lie algebras}\\\hline
&$e^{24}-e^{13}+e^{56}$&\\\cline{2-2}
$37A$&$-e^{125}+e^{236}+e^{146}-e^{345}$&\\\cline{2-2}
&$e^7$&\\\hline
&$2(e^{13}+e^{24}-e^{67})$&\\\cline{2-2}
$37B$&$2\sqrt{2}(e^{127}-e^{146}+e^{236}-e^{347})$&\\\cline{2-2}
&$e^5+e^7$&\\\hline
&$e^{12}-e^{34}-e^{67}$&\\\cline{2-2}
$37C$&$e^{236}-e^{247}-e^{137}-e^{146}$&\\\cline{2-2}
&$e^5$&\\\hline
&$-e^{12}+e^{34}-e^{67}$&\\\cline{2-2}
$37D$&$-e^{136}-e^{146}+e^{137}$&\\
&$-e^{147}-e^{246}+e^{237}$&\\\cline{2-2}
&$e^5$&\\\hline
&$e^{12}+e^{34}+e^{56}$&\\\cline{2-2}
$17$&$-e^{246}+e^{136}+e^{145}+e^{235}$&\\\cline{2-2}
&$e^7$&\\\hline
&$-e^{12}-e^{34}+e^{67}$&\\\cline{2-2}
$37B_1$&$-e^{146}+e^{137}-e^{247}-e^{236}$&\\\cline{2-2}
&$e^5$&\\\hline
&$e^{12}-e^{34}-e^{67}$&\\\cline{2-2}
$37D_1$&$e^{136}-e^{147}+e^{237}+e^{246}$&\\\cline{2-2}
&$e_5$&\\\hline
\multicolumn{3}{|l|}{$3$-step nilpotent Lie algebras}\\\hline
&$e^{25}+e^{14}-e^{36}$&\\\cline{2-2}
$357A$&$e^{123}-e^{246}+e^{156}-e^{345}$&\\\cline{2-2}
&$e^7$&\\\hline
&$AB(e^{12}+2e^{34}+e^{56})$&\\\cline{2-2}
$257A$&$B^2(\frac{1}{10}e^{125}-\frac{2}{5}e^{134}+4e^{136}+e^{145}$&$B=\frac{5A^3}{\sqrt{390}},$\\
&$+\frac{1}{5}e^{156}+4e^{235}-4e^{246}-\frac{1}{5}e^{345})$&$A>\sqrt{1791}$\\\cline{2-2}
&$A^2(30e^6+e^7)$&\\\hline
&$16(e^{13}+e^{25}+e^{46})$&\\\cline{2-2}
$257C$&$64(e^{124}-e^{156}+e^{236}+e^{345})$&\\\cline{2-2}
&$4(-e^3+e^7)$&\\\hline
&$-e^{12}+e^{35}+e^{45}+e^{46}$&\\\cline{2-2}
$257I$&$-e^{134}+e^{156}-e^{236}+e^{245}-e^{246}$&\\\cline{2-2}
&$e^7$&\\\hline
&$16(e^{13}+e^{25}+e^{46})$&\\\cline{2-2}
$257J$&$64(-e^{124}+e^{156}-e^{236}-e^{345})$&\\\cline{2-2}
&$4(e^3+e^7)$&\\\hline
&$e^{12}+e^{34}+e^{56}$&\\\cline{2-2}
$247A$&$e^{135}-e^{146}-e^{236}-e^{245}$&\\\cline{2-2}
&$e^7$&\\\hline
&$AB(e^{12}+e^{26}-e^{35}+e^{45}+e^{46})$&\\\cline{2-2}
&$B^2(-e^{123}+e^{124}-2e^{125}-e^{126}+$&$B=\frac{8}{25}A^3,$\\
$247B$&$+e^{134}-e^{135}-e^{136}-e^{145}+$&$0<A<\sqrt{\frac{7}{40}}$\\
&$+e^{156}-e^{235}-e^{236}-e^{245})$&\\\cline{2-2}
&$A^2(3e^5+e^6+e^7)$&\\\hline
&$16(e^{13}+e^{25}+e^{47})$&\\\cline{2-2}
$247C$&$64(e^{124}-e^{157}+e^{237}+e^{345})$&\\\cline{2-2}
&$4(e^4+e^6)$&\\\hline
&$e^{12}-e^{34}-e^{56}$&\\\cline{2-2}
$247D$&$e^{136}+e^{145}-e^{235}+e^{246}$&\\\cline{2-2}
&$e_7$&\\\hline
&$\frac{1}{16}(e^{15}+e^{16}-e^{25}-e^{34}+e^{56})$&\\\cline{2-2}
$247F$&$\frac{1}{64}(e^{123}-e^{124}+e^{145}+$&\\
&$2e^{235}+e^{236}+e^{246}-e^{356})$&\\\cline{2-2}
&$\frac{1}{4}(-e^4+e^7)$&\\\hline
&$\frac{1}{16}(e^{12}+e^{15}-e^{23}+e^{46})$&\\\cline{2-2}
$247I$&$64(e^{126}-e^{134}+$&\\
&$+e^{236}+e^{245}-e^{345})$&\\\cline{2-2}
&$\frac{1}{4}(e^4+e^7)$&\\\hline
&$e^{13}-e^{25}-e^{46}$&\\\cline{2-2}
$247J$&$-e^{124}+e^{126}-e^{145}+$&\\
&$+e^{236}+2e^{345}+e^{356}$&\\\cline{2-2}
&$e^7$&\\\hline
&$AB(-e^{13}+e^{15}-e^{25}+$&\\
&$+2e^{27}-e^{34}+e^{35}+$&\\
&$+2e^{45}-5e^{47})$&\\\cline{2-2}
$247L$&$B^2(e^{123}+e^{124}-e^{127}+$&$B=A^3,$\\
&$2e^{134}-3e^{135}+8e^{137}+2e^{147}+$&$0<A<\frac{1}{\sqrt{452}}$\\
&$-e^{157}-e^{234}-e^{237}-e^{345})$&\\\cline{2-2}
&$A^2(6e^4+e^6+4e^7)$&\\\hline
&$AB(e^{15}+e^{23}+e^{35}-2e^{46})$&\\\cline{2-2}
$247M$&$B^2(-2e^{124}-2e^{126}+$&$B=2A^3,$\\
&$-e^{134}-2e^{136}-e^{145}+$&$0<A<\frac{1}{\sqrt{32}}$\\
&$+e^{156}+e^{234}+e^{236}+e^{245}+e^{345})$&\\\cline{2-2}
&$A^2(4e^5+e^7)$&\\\hline
&$-e^{14}-e^{25}-e^{37}$&\\\cline{2-2}
$247N$&$e^{123}-e^{157}+e^{247}-e^{345}$&\\\cline{2-2}
&$e^6$&\\\hline
&$3\sqrt{3}(-e^{26}+e^{34}+e^{12}+e^{15})$&\\\cline{2-2}
$247O$&$3(-e^{123}+e^{146}+e^{236}+$&\\
&$+e^{245}-e^{356})$&\\\cline{2-2}
&$e^4-e^6+e^7$&\\\hline

&$\frac{1}{16}(e^{17}-e^{23}+e^{25}+$&\\
&$-e^{27}+e^{34}+e^{47})$&\\\cline{2-2}
$247P$&$\frac{1}{64}(e^{123}+2e^{134}+2e^{145}+$&\\
&$-e^{147}-e^{245}+e^{357})$&\\\cline{2-2}
&$\frac{1}{4}(e^6-e^7)$&\\\hline
&$e^{14}+e^{23}+e^{57}$&\\\cline{2-2}
$247Q$&$e^{125}-e^{137}+e^{247}+e^{345}$&\\\cline{2-2}
&$e^6$&\\\hline
&$2e^{12}+e^{15}+e^{26}+e^{34}+e^{56}$&\\\cline{2-2}
$157$&$e^{124}+e^{135}-e^{145}+$&\\
&$-e^{146}-e^{235}-e^{236}-e^{245}$&\\\cline{2-2}
&$e^7$&\\\hline
&$AB(\frac{1}{2}e^{12}+e^{34}+\frac{1}{2}e^{56}$&$B=\frac{1}{4}A^3$\\\cline{2-2}
$147A$&$B^2(-e^{123}+2e^{136}+e^{145}+$&\\
&$+e^{235}-e^{246}+e^{356})$&$A>5$\\\cline{2-2}
&$A^2(\frac{5}{2}e^4+e^7)$&\\\hline
&$AB(-e^{13}+4e^{24}+e^{56})$&\\\cline{2-2}
&$B^2(e^{123}+2e^{124}-2e^{125}+4e^{126}+$&$B=\frac{8}{\sqrt{11}}A^3,$\\
$147B$&$+e^{145}+\frac{1}{2}e^{146}-\frac{1}{2}e^{156}-2e^{234}+$&$A>\sqrt{19}$\\
&$+2e^{235}+1e^{256}+1e^{346}-\frac{1}{2}e^{356})$&\\\cline{2-2}
&$A^2(-4e^4+6e^5+e^7)$&\\\hline
&$AB(-e^{12}-e^{16}-e^{23}+$&\\
&$-e^{25}+e^{34}+e^{36}+e^{45})$&\\\cline{2-2}
$147D$&$B^2(-e^{123}-2e^{134}-e^{136}+$&$B=A^3,$\\
&$+e^{145}-e^{246}-e^{356})$&$0<A<\frac{1}{\sqrt{5}}$\\\cline{2-2}
&$A^2(e^4-e^6+e^7)$&\\\hline
&$\frac{2}{3}\sqrt{3}(e^{13}+$&\\
&$-e^{24}+e^{56})$&\\\cline{2-2}
$137A$&$\frac{4}{3}(e^{126}-e^{145}+$&\\
&$-e^{235}-e^{345}+e^{346})$&\\\cline{2-2}
&$e^7$&\\\hline
&$e^{13}-e^{24}+e^{56}$&\\\cline{2-2}
$137B$&$-e^{126}+e^{145}+$&\\
&$+e^{235}+e^{345}-e^{346}$&\\\cline{2-2}
&$e^7$&\\\hline
&$\frac{2}{3}\sqrt{3}(e^{13}-e^{24}+e^{56})$&\\\cline{2-2}
$137C$&$\frac{4}{3}(-e^{126}+e^{145}+e^{235}-e^{346})$&\\\cline{2-2}
&$e^7$&\\\hline
&$16(e^{15}+e^{23}+e^{35}-e^{46})$&\\\cline{2-2}
$257J_1$&$64(e^{124}+e^{136}-e^{234}-e^{256}-e^{345})$&\\\cline{2-2}
&$4(e^3+e^6+e^7)$&\\\hline
&$e^{23}-e^{16}+e^{45}$&\\\cline{2-2}
$247F_1$&$-e^{125}+e^{246}+e^{134}-e^{356}$&\\\cline{2-2}
&$e^7$&\\\hline
&$e^{23}+e^{17}+e^{45}$&\\\cline{2-2}
$247P_1$&$e^{125}+e^{247}+e^{134}-e^{357}$&\\\cline{2-2}
&$e^6$&\\\hline
&$e^{16}+e^{23}+e^{45}$&\\\cline{2-2}
$147A_1$&$e^{124}-e^{135}-e^{256}-e^{346}$&\\\cline{2-2}
&$e^7$&\\\hline
&$\frac{16}{\sqrt{3}}(e^{12}+2e^{34}+2e^{56}$&\\\cline{2-2}
$137A_1$&$\frac{16^2}{3}(e^{135}+e^{136}-e^{145}-\frac{3}{2}e^{146}+$&\\
&$-e^{236}-\frac{1}{2}e^{245}+e^{246})$&\\\cline{2-2}
&$e^7$&\\\hline
&$AB(e^{12}+e^{34}+e^{56})$&\\\cline{2-2}
$137B_1$&$B^2(e^{135}+2e^{136}+e^{145}+$&$B=A^3,$\\
&$+\frac{1}{2}e^{235}-e^{236}-2e^{245}-e^{246})$&$A>\sqrt{10}$\\\cline{2-2}
&$A^2(-2e^6+e^7)$&\\\hline
&$AB(e^{13}-2e^{26}-e^{34}+e^{45})$&\\\cline{2-2}
$147E(\muup)$&$B^2(-e^{123}-2e^{125}-e^{146}+$&$B=A^3,$\\
$\muup\neq 0,1$&$e^{245}+e^{356})$&$(A-3+2\muup)(A+3-2\muup)>0$\\\cline{2-2}
&$A^2((3-2\muup)e^6+e^7)$&\\\hline
&$AB(e^{13}+e^{26}+e^{34}-e^{45})$&\\\cline{2-2}
$147E_1(\muup)$&$B^2(e^{123}+2e^{125}+e^{146}+e^{245}+e^{356})$&$B=A^3,$\\\cline{2-2}
$\muup>1$&$A^2(5e^6+e^7)$&$A>5$\\\hline
\multicolumn{3}{|l|}{$4$-step nilpotent Lie algebras}\\\hline
&$e^{15}+e^{24}+e^{36}$&\\\cline{2-2}
$2457A$&$e^{123}-e^{146}+e^{256}+e^{345}$&\\\cline{2-2}
&$e^7$&\\\hline
&$2(-e^{24}+e^{15}-e^{36})$&\\\cline{2-2}
$2457B$&$2\sqrt{2}(e^{123}-e^{146}-e^{256}-e^{345})$&\\\cline{2-2}
&$-e^6+e^7$&\\\hline
&$-e^{26}-e^{46}+e^{23}+e^{15}$&\\\cline{2-2}
$2457C$&$e^{123}+e^{235}+e^{124}+e^{245}+$&\\
&$+e^{134}-e^{345}-e^{256}+e^{136}$&\\\cline{2-2}
&$e^7$&\\\hline
&$16(e^{15}-e^{23}+e^{46})$&\\\cline{2-2}
$2457D$&$64(e^{124}+e^{126}+e^{136}+$&\\
&$+e^{245}+e^{256}+e^{345})$&\\\cline{2-2}
&$4(e^3+e^7)$&\\\hline
&$16(e^{15}-e^{23}+e^{47})$&\\\cline{2-2}
$2457E$&$64(e^{124}+e^{137}+e^{257}+e^{345})$&\\\cline{2-2}
&$4(-e^4+e^6)$&\\\hline
&$e^{15}+e^{23}-e^{56}-e^{46}$&\\\cline{2-2}
$2457F$&$e^{125}+e^{124}+e^{136}-e^{256}-e^{345}$&\\\cline{2-2}
&$e^7$&\\\hline
&$e^{15}-e^{23}+e^{47}$&\\\cline{2-2}
$2457G$&$-e^{124}-e^{137}-e^{257}-e^{345}$&\\\cline{2-2}
&$e^6$&\\\hline
&$-e^{27}-e^{47}+e^{23}+e^{15}$&\\\cline{2-2}
$2457H$&$-e^{123}+e^{235}+e^{124}+e^{245}+$&\\
&$+e^{134}-e^{345}-e^{257}+e^{137}$&\\\cline{2-2}
&$e^6$&\\\hline
&$3(e^{15}+e^{23}-e^{56}-e^{46})$&\\\cline{2-2}
$2457I$&$3\sqrt{3}(e^{124}+e^{125}+$&\\
&$+e^{136}-e^{256}-e^{345})$&\\\cline{2-2}
&$e^4+e^7$&\\\hline
&$3(e^{15}+e^{23}-e^{56}-e^{46})$&\\\cline{2-2}
$2457J$&$3\sqrt{3}(e^{124}+e^{125}+$&\\
&$+e^{136}-e^{256}-e^{345})$&\\\cline{2-2}
&$e^4+e^7$&\\\hline
&$-e^{27}-e^{47}+e^{23}+e^{15}$&\\\cline{2-2}
$2457K$&$-e^{123}+e^{235}+e^{124}+e^{245}+$&\\
&$+e^{134}-e^{345}-e^{257}+e^{137}$&\\\cline{2-2}
&$e^6$&\\\hline
&$5(-e^{14}+e^{24}+e^{25}+e^{36})$&\\\cline{2-2}
$2457L$&$5\sqrt{5}(-e^{123}+e^{156}+$&\\
&$+e^{246}-e^{256}-e^{345})$&\\\cline{2-2}
&$-2e^6+e^7$&\\\hline
&$e^{15}+e^{24}+e^{37}$&\\\cline{2-2}
$2457M$&$e^{123}-e^{147}+e^{257}+e^{345}$&\\\cline{2-2}
&$e^6$&\\\hline
&$e^{13}-e^{24}+e^{57}$&\\\cline{2-2}
$2357A$&$-e^{125}-e^{237}-e^{147}-e^{345}$&\\\cline{2-2}
&$-\frac{1}{2}e^5+e^6$&\\\hline
&$e^{13}-e^{24}+e^{57}$&\\\cline{2-2}
$2357B$&$-e^{125}-e^{237}-e^{147}-e^{345}$&\\\cline{2-2}
&$e^6$&\\\hline
&$e^{13}-e^{24}+e^{57}$&\\\cline{2-2}
$2357C$&$-e^{125}-e^{237}-e^{147}-e^{345}$&\\\cline{2-2}
&$e^6$&\\\hline
&$e^{13}-e^{24}+e^{57}$&\\\cline{2-2}
$2357D$&$-e^{125}-e^{237}-e^{147}-e^{345}$&\\\cline{2-2}
&$e^6$&\\\hline
&$-e^{12}+e^{34}-e^{56}$&\\\cline{2-2}
$1357A$&$e^{135}+e^{136}+$&\\
&$-e^{145}+e^{146}+e^{235}+e^{246}$&\\\cline{2-2}
&$e^7$&\\\hline
&$AB(-e^{13}-e^{24}+\frac{1}{4}e^{45}-e^{56})$&\\\cline{2-2}
$1357B$&$B^2(2e^{126}+e^{145}+$&$B=A^3,$\\
&$+\frac{1}{2}e^{156}-e^{235}-\frac{1}{2}e^{346})$&$A>\sqrt{3}$\\\cline{2-2}
&$A^2(-4e^4+e^7)$&\\\hline
&$16(-e^{12}+2e^{34}-e^{36}+e^{45})$&\\\cline{2-2}
$1357C$&$32(-2e^{124}-3e^{134}+4e^{145}+$&\\&$+2e^{156}-4e^{235}-2e^{246}-2e^{346})$&\\\cline{2-2}
&$4(e^4+e^5+e^7)$&\\\hline
&$e^{12}-e^{34}-e^{56}$&\\\cline{2-2}
$1357D$&$e^{136}+e^{145}-e^{146}+$&\\
&$-e^{235}+e^{236}+e^{246}$&\\\cline{2-2}
&$e^7$&\\\hline
&$256(e^{12}+e^{13}+e^{35}-e^{46}+e^{56})$&\\\cline{2-2}
&$4096(e^{123}-e^{124}+e^{125}-2e^{126}+$&\\
$1357F$&$+e^{136}+e^{145}-e^{234}+e^{236}+$&\\
&$+e^{245}-e^{246}-e^{256})$&\\\hline
&$16(2e^5+e^7)$&\\\cline{2-2}
&$2(e^{12}-e^{34}-2e^{56})$&\\\cline{2-2}
$1357G$&$4e^{136}+4e^{145}-4e^{146}-4e^{235}+$&\\
&$-4e^{236}-4e^{245}+9e^{246}$&\\\cline{2-2}
&$e^7$&\\\hline
&$2(-e^{12}+2e^{34}-e^{56})$&\\\cline{2-2}
$1357H$&$4(e^{135}+e^{145}+e^{146}+$&\\
&$+e^{235}-e^{236}+2e^{245})$&\\\cline{2-2}
&$e^7$&\\\hline
&$3(e^{15}+e^{24}+e^{36})$&\\\cline{2-2}
$1357I$&$3\sqrt{3}(e^{123}-e^{146}+e^{256}+e^{345})$&\\\cline{2-2}
&$e^5+e^6+e^7$&\\\hline
&$81(-e^{13}+2e^{23}+e^{25}-e^{46}+e^{56})$&\\\cline{2-2}
$1357J$&$729(2e^{123}+3e^{124}-2e^{125}+e^{136}$&\\
&$+e^{146}-e^{256}-e^{345})$&\\\cline{2-2}
&$9(2e^3+3e^4-2e^5-e^6+e^7)$&\\\hline
&$\frac{27}{2}\sqrt{6}(e^{12}+e^{34}+e^{56})$&\\\cline{2-2}
$1357L$&$\frac{81}{2}(6e^{135}-2e^{136}-3e^{146}-3e^{236}$&\\&
$-9e^{245}+3e^{246})$&\\\cline{2-2}
&$9(-e^6+e^7)$&\\\hline
&$e^{12}-e^{34}-e^{56}$&\\\cline{2-2}
$1357O$&$e^{136}+\frac{5}{4}e^{145}-\frac{1}{2}e^{146}-e^{235}+$&\\
&$-\frac{1}{2}e^{245}+e^{246}$&\\\cline{2-2}
&$e^7$&\\\hline
&$\frac{1}{\sqrt{2}}(e^{12}+e^{34}+e^{56})$&\\\cline{2-2}
$1357P$&$\frac{1}{2}(-e^{135}+e^{146}-e^{235}+$&\\
&$+e^{236}+2e^{245}+e^{246})$&\\\cline{2-2}
&$e^7$&\\\hline
&$e^{12}+e^{34}+e^{56}$&\\\cline{2-2}
$1357Q$&$-e^{135}+e^{146}-e^{235}+$&\\
&$+e^{236}+e^{245}+e^{246}$&\\\cline{2-2}
&$e^7$&\\\hline
&$-e^{12}-e^{34}+e^{56}$&\\\cline{2-2}
$1357R$&$-e^{135}-e^{136}+e^{145}+$&\\
&$+e^{235}+e^{245}+e^{246}$&\\\cline{2-2}
&$e^7$&\\\hline
&$e^{15}+e^{24}+e^{36}$&\\\cline{2-2}
$2457L_1$&$-e^{123}-e^{235}-e^{134}+$&\\
&$-2e^{345}-e^{256}+e^{146}$&\\\cline{2-2}
&$e^7$&\\\hline
&$e^{125}-e^{237}+e^{147}-e^{345}$&\\\cline{2-2}
$2357D_1$&$e^{13}+e^{24}-e^{57}$&\\\cline{2-2}
&$e^6$&\\\hline
&$16(e^{12}+e^{35}-e^{46}+e^{56})$&\\\cline{2-2}
$1357F_1$&$64(-e^{134}+e^{136}+e^{145}-e^{146}+$&\\
&$-e^{234}-e^{245}-e^{246}-e^{256})$&\\\cline{2-2}
&$4(-e^4+e^6+e^7)$&\\\hline
&$2(e^{12}+2e^{34}+2e^{56})$&\\\cline{2-2}
$1357P_1$&$4(-e^{136}-e^{145}-e^{235}+e^{246})$&\\\cline{2-2}
&$e^7$&\\\hline
&$2(-e^{12}-e^{34}+2e^{56})$&\\\cline{2-2}
$1357Q_1$&$4(-e^{135}+\frac{17}{4}e^{145}-e^{146}+e^{235}+$\\
&$-e^{236}-e^{245}+\frac{9}{4}e^{246})$&\\\cline{2-2}
&$e^7$&\\\hline
&$\frac{(-\muup(\muup+1))^\frac{3}{2}}{\muup^2(\muup+1)^2}(-e^{12}-e^{34}+e^{56})$&\\\cline{2-2}
$1357M(\muup)$&$-\frac{1}{\muup(\muup+1)^4}[(\muup+1)^3e^{135}+e^{146}+$&\\
$\muup(\muup+1)<0$&$-e^{235}-\muup e^{236}-(\muup+1)e^{245}-e^{246}]$&\\\cline{2-2}
&$e^7$&\\\hline
&$AB(e^{12}-e^{34}+e^{36}+e^{45}+\frac{1}{\muup+1}e^{56})$&\\\cline{2-2}
&$B^2(e^{135}-e^{146}+(2\muup^2+3\muup+1)e^{234}+$&\\
$1357M(\muup)$&$+(-\frac{\muup}{\muup+1}e^{235}-\muup e^{236}-(\muup+1)e^{245}+$&$B=\frac{(\muup(\muup+1)^\frac{3}{2})}{\muup(\muup+1)^3}A^3,$\\
$\muup(\muup+1)>0$&$+\frac{\muup}{\muup+1}e^{246})$&$A>\frac{\sqrt{2\muup^4+6\muup^3+9\muup^2+6\muup+1}(\muup+1)^3(2\muup+1)}{(\muup^2+2\muup+2)\muup^2}$\\\cline{2-2}
&$A^2(\frac{2(\muup+1)^3(\muup+1\frac{1}{2})}{\muup(\muup^2+2\muup+2)}e^5-\frac{2(\muup+1)^2(\muup+\frac{1}{2})}{\muup(\muup^2+2\muup+2)}e^6+e^7)$&\\\hline
&$AB(-(\muup+2)e^{12}-e^{34}+\frac{1}{\muup+2}e^{56})$&\\\cline{2-2}
&$B^2(-\muup-2)^{-\frac{3}{2}}(-2e^{126}+2e^{134}+\frac{4}{\muup+2}e^{135}+$&$\psi_+\wedge\psi_-=\frac{2}{3}\omega^3,$\\
$1357N(\muup)$&$+2e^{146}+\frac{2}{\muup+2}e^{156}-\frac{2\muup}{\muup+2}e^{236}+$&$A>\frac{-2\muup+2}{\sqrt{-\muup-2}}$\\
$\muup+2<0$&$-2e^{245}+\sqrt{(-\muup-2)}e^{246}+\frac{2}{\muup+2}e^{346})$&\\\cline{2-2}
&$A^2((-\muup^2-\muup+2)e^3+e^7)$&\\\hline
&$AB(e^{14}+e^{23}+e^{56})$&\\\cline{2-2}
&$B^2(e^{123}-\frac{1}{\muup+1}e^{125}+C_1e^{126}-2e^{135}+$&$d(\frac{1}{2}\omega^2)=\psi_-\wedge d(\eta),$\\
$1357N(\muup)$&$+(\muup+1)e^{136}-\muup e^{146}-e^{156}+$&$\psi_+\wedge\psi_-=\frac{2}{3}\omega^3,$\\
$\muup+2>0,\muup\neq -1$&$+\muup e^{236}+(\muup+2)e^{245}-e^{346})$&$\mathrm{det}(g)>0$\\\cline{2-2}
&$A^2(C_2e^3+\frac{1}{\muup+1}e^5+e^7)$&\\\hline
&$AB(e^{12}-5e^{14}-2e^{15}-e^{23}+$&\\
&$-3e^{34}-e^{35}-e^{46})$&\\\cline{2-2}
&$B^2(2e^{123}-2e^{124}-3e^{126}+e^{134}+$&$B=A^3,$\\
$1357N(\muup)$&$+2e^{135}+e^{136}+3e^{146}+e^{156}+$&$A>\sqrt{\frac{119}{2}}$\\
$\muup=-1$&$+e^{234}+e^{236}-e^{245}+e^{346})$&\\\cline{2-2}
&$A^2(\frac{1}{2}e^2-e^5+3e^6+e^7)$&\\\hline
&$\muup\sqrt{\frac{(\muup-1)^2}{(\muup+2)^2}}(e^{12}+$&\\
&$+\muup e^{34}+e^{56})$&\\\cline{2-2}
$1357S(\muup)$&$\muup^2\frac{(\muup-1)^2}{(\muup+2)^2}(\frac{3}{\muup-1}e^{135}+$&\\
$\muup>1$&$+e^{136}+e^{145}+e^{235}+$&\\
&$-\frac{\muup+2}{\muup-1}e^{236}-e^{245}-e^{246})$&\\\cline{2-2}
&$e^7$&\\\hline
&$-\left(\frac{\muup-2}{\muup-1}\right)^{\frac{3}{2}}\frac{(\muup-1)^2}{\muup-2}(-e^{12}+$&\\
&$+(-\muup+2)e^{34}-e^{56})$&\\\cline{2-2}
$1357S(\muup)$&$\left(\frac{\muup-2}{\muup-1}\right)^{2}\left(\frac{(\muup-1)^2}{\muup-2}\right)^2(-\frac{1}{\muup-1}e^{135}+$&\\
$\muup<1$&$-e^{136}+e^{145}-e^{235}+$&\\
&$-\frac{\muup-2}{\muup-1}e^{236}+e^{245}-e^{246})$&\\\cline{2-2}
&$e^7$&\\\hline
&$\sqrt{\frac{\muup^2}{\muup^2+1}(\frac{1}{\muup}e^{12}+\muup e^{34}+e^{56})}$&\\\cline{2-2}
$1357QRS_1(\muup)$&$\muup^2(\muup^2+1)(-\frac{1}{\muup}e^{135}+2e^{145}+e^{146}+$&\\
$\muup\neq 0$&$+e^{235}+e^{236}+\frac{1}{\muup}e^{245}-\muup e^{246})$&\\\cline{2-2}
&$e^7$&\\\hline
\multicolumn{3}{|l|}{$5$-step nilpotent Lie algebras}\\\hline
&$\frac{1}{\sqrt{4}}(e^{16}-2e^{25}-\frac{1}{4}e^{34}+4e^{56})$&\\\cline{2-2}
$23457C$&$\frac{1}{4}(-\frac{3}{4}e^{124}+e^{135}+e^{146}+e^{236}-e^{245})$&\\\cline{2-2}
&$e^7$&\\\hline
&$\sqrt{\frac{3}{4}}(-e^{12}-\frac{1}{2}e^{34}+2e^{56})$&\\\cline{2-2}
$23457D$&$\frac{3}{4}(e^{135}+e^{146}-e^{235}+e^{236}-e^{245})$&\\\cline{2-2}
&$e^7$&\\\hline
&$AB(e^{13}+e^{24}-2e^{35}+$&\\
&$+2e^{37}-e^{45}-e^{57})$&$B=2A^3,$\\\cline{2-2}
$23457E$&$B^2(\frac{1}{2}e^{124}+2e^{125}-2e^{127}+e^{134}+$&$0<A<\frac{\sqrt{2}}{4}$\\
&$+e^{145}-e^{157}-e^{235}+e^{245}-e^{247})$&\\\cline{2-2}
&$A^2(2e^5+e^6)$&\\\hline
&$AB(e^{15}+e^{24}-e^{34}+e^{36}+e^{45}+2e^{56})$&\\\cline{2-2}
$23457G$&$B^2(-3e^{123}-3e^{125}+e^{135}+$&$B=A^3,$\\
&$+e^{146}+e^{236}-e^{245})$&$A>\sqrt{11}$\\\cline{2-2}
&$A^2(-3e^4+2e^6+e^7)$&\\\hline
&$\sqrt{\frac{3}{4}}(-e^{12}-\frac{1}{2}e^{34}+2e^{56})$&\\\cline{2-2}
$13457D$&$\frac{3}{4}(e^{135}+e^{146}-e^{235}+e^{236}-e^{245})$&\\\cline{2-2}
&$e^7$&\\\hline
&$16(e^{12}-e^{35}+e^{46})$&\\\cline{2-2}
$13457F$&$64(e^{134}+e^{136}+e^{145}+$&\\
&$+e^{156}-e^{236}-e^{245})$&\\\cline{2-2}
&$4(-e^4+e^5+e^6+e^7)$&\\\hline
&$81(e^{15}-e^{23}+e^{24}+e^{26}+$&\\&$-e^{34}-3e^{36}-e^{46})$&\\\cline{2-2}
$12457A$&$27(-e^{124}+e^{126}-e^{136}+e^{146}+$&\\
&$-e^{245}-2e^{256}-e^{345})$&\\\cline{2-2}
&$3(e_1+e_3+e_5-e_6+e_7)$&\\\hline
&$AB(e^{15}+e^{23}+2e^{24}+2e^{26}+$&\\
&$+e^{34}+2e^{36}-e^{46})$&\\\cline{2-2}
$12457B$&$B^2(-e^{123}-e^{134}+e^{146}+$&$B=3A^3,$\\
&$-2e^{256}-e^{345})$&$A>\frac{\sqrt{26}}{3}$\\\cline{2-2}
&$A^2(-2e^3+e^6+e^7)$&\\\hline
&$AB(-4e^{12}-e^{15}-e^{24}+4e^{26}+$&\\
&$+2e^{36}+e^{46}+e^{56})$&\\\cline{2-2}
&$B^2(e^{123}+\frac{5}{2}e^{124}+\frac{5}{2}e^{125}+$&\\
$12457C$&$+e^{126}+e^{134}+e^{135}+e^{146}+$&$B=2A^3,$\\
&$+e^{234}-e^{235}-\frac{1}{2}e^{256}+\frac{1}{2}e^{345})$&$A>\frac{\sqrt{1097}}{4}$\\\cline{2-2}
&$A^2(5e^3-4e^4+2e^5+e^7)$&\\\hline
&$AB(2e^{13}-3e^{15}-4e^{26}+e^{35}+$&\\
&$+2e^{36}+2e^{45}+2e^{46}+e^{56})$&\\\cline{2-2}
&$B^2(-6e{124}+4e^{125}+3e^{126}+$&$B=16A^3,$\\
$12457D$&$+3e^{134}-e^{135}-2e^{145}+e^{146}+$&$A>\frac{\sqrt{2478}}{2}$\\
&$-2e^{234}-\frac{1}{2}e^{256}+\frac{1}{2}e^{345})$&\\\cline{2-2}
&$A^2(-140e^1+52e^3-14e^4-8e^6+e^7)$&\\\hline
&$AB(-2e^{12}+e^{15}+2e^{15}+$&\\
&$+e^{24}-e^{25}-e^{46})$&\\\cline{2-2}
$12457E$&$B^2(e^{124}+e^{134}+e^{146}-e^{156}+$&$B=A^3,$\\
&$-e^{236}-e^{245}-2e^{256}-e^{345}$)&$A>2\sqrt{5}$\\\cline{2-2}
&$A^2(\frac{3}{2}e^3+\frac{1}{2}e^4+4e^5+2e^6+e^7)$&\\\hline
&$AB(2e^{13}-4e^{15}-e^{24}+$&\\
&$-2e^{35}+2e^{36}-e^{56})$&\\\cline{2-2}
$12457F$&$B^2(e^{123}+2e^{126}+2e^{134}+e^{146})$&$B=6A^3,$\\
&$-e^{235}-e^{236}-\frac{1}{2}e^{256}+\frac{1}{2}e^{345})$&$A>\sqrt{\frac{5}{3}}$\\\cline{2-2}
&$A^2(2e^4+e^7)$&\\\hline
&$AB(-4e^{13}-e^{24}+2e^{36}-2e^{56})$&\\\cline{2-2}
$12457G$&$B^2(-e^{125}+e^{126}+e^{134}+e^{145}+e^{146}+$&$B=8A^3,$\\
&$-e^{235}-e^{236}-\frac{1}{2}e^{256}+\frac{1}{2}e^{345})$&$A>\frac{\sqrt{654}}{8}$\\\cline{2-2}
&$A^2(-6e^3+6e^4+3e^5+2e^6+e^7)$&\\\hline
&$-e^{12}-e^{34}+e^{56}$&\\\cline{2-2}
$12457H$&$-2e^{135}+e^{136}+e^{145}-e^{146}+$&\\
&$-e^{236}+e^{245}$&\\\cline{2-2}
&$e^7$&\\\hline
&$\frac{1}{2}(-e^{12}-e^{34}+e^{56})$&\\\cline{2-2}
$12457I$&$\frac{1}{4}(-e^{135}-e^{146}-2e^{236}+2e^{245})$&\\\cline{2-2}
&$e^7$&\\\hline
&$2(-e^{12}+2e^{34}-e^{56})$&\\\cline{2-2}
$12457J$&$4(e^{135}+e^{146}-e^{236}+e^{245})$&\\\cline{2-2}
&$e^5+e^7$&\\\hline
&$16(-e^{12}+e^{34}-e^{56})$&\\\cline{2-2}
$12457K$&$64(e^{135}+e^{146}-e^{236}+e^{245})$&\\\cline{2-2}
&$4(e^5+e^7)$&\\\hline
&$AB(e^{12}+e^{34}+\frac{1}{4}e^{56})$&\\\cline{2-2}
&$B^2(-\frac{1}{2}e^{125}+e^{134}-e^{135}+7e^{136}+$&$B=\frac{2}{\sqrt{139}}A^3,$\\
$12457L$&$+7e^{145}+\frac{3}{4}e^{146}-\frac{1}{4}e^{156}-3e^{234}+$&$A>\frac{\sqrt{429701402}}{4}$\\
&$+10e^{235}-\frac{1}{4}e^{246}+\frac{3}{4}e^{256}+\frac{1}{2}e^{345}$&\\\cline{2-2}
&$A^2(-\frac{4169}{2}e^3-\frac{247}{4}e^4+$&\\
&$+\frac{565}{4}e^5+\frac{27}{2}e^6+e^7)$&\\\hline
&$AB(e^{13}+e^{26}-e^{45}+e^{56})$&\\\cline{2-2}
$12357A$&$B^2(-e^{124}-2e^{125}+e^{145}-e^{146}+$&$B=A^3,$\\
&$-2e^{234}-e^{235}-e^{236}-e^{345})$&$A>\sqrt{7}$\\\cline{2-2}
&$A^2(3e^4+e^5+e^6+e^7)$&\\\hline
&$e^{13}-e^{24}+e^{56}$&\\\cline{2-2}
$12357B$&$e^{126}-e^{235}+e^{346}-e^{145}$&\\\cline{2-2}
&$-\frac{1}{2}e^5+e^7$&\\\hline
&$AB(e^{13}-e^{24}+e^{56})$&\\\cline{2-2}
$12357C$&$B^2(-2e^{125}+e^{145}-e^{146}-e^{235}+$&$B=A^3,$\\
&$-e^{236}-e^{345})$&$A>3$\\\cline{2-2}
&$A^2(2e^4+e^6+e^7)$&\\\hline
&$AB(e^{13}+e^{24}+e^{34}-e^{56})$&\\\cline{2-2}
$12457J_1$&$B^2(4e^{125}+e^{126}+2e^{135}+e^{136}+$&$B=\frac{1}{2}A^3,$\\
&$+e^{145}+e^{146}-2e^{236}+2e^{245})$&$A>3\sqrt{6}$\\\cline{2-2}
&$A^2(\frac{5}{2}e^4+2e^5+e^6+e^7)$&\\\hline
&$AB(e^{12}+e^{34}+e^{56})$&\\\cline{2-2}
&$B^2(\frac{1}{2}e^{125}+9e^{135}+2e^{136}-2e^{145}+$&$B=\frac{2}{\sqrt{5}}A^3,$\\
$12457L_1$&$-\frac{1}{2}e^{146}-e^{234}-2e^{236}-2e^{245}+$&$A>5\sqrt{3}$\\
&$+e^{256}-\frac{1}{2}e^{345})$&\\\cline{2-2}
&$8e^3-4e^4-6e^5-2e^6+e^7$&\\\hline
&$32(-e^{12}-2e^{24}+e^{26}+$&\\
&$+2e^{36}-e^{45}-2e^{56})$&\\\cline{2-2}
$12457N_1$&$128(2e^{123}-e^{125}+2e^{135}+$&\\
&$+e^{146}-e^{234}-4e^{256}+2e^{245})$&\\\cline{2-2}
&$6e^4-2e^5+4e^6+4e^7$&\\\hline
&$e^{13}-e^{24}+e^{45}+e^{56}$&\\\cline{2-2}
$12357B_1$&$-e^{125}-e^{146}+e^{234}-e^{236}-e^{345}$&\\\cline{2-2}
&$e^7$&\\\hline

&$AB(e^{15}+e^{23}-e^{36}-e^{46})$&\\\cline{2-2}
$12457N(\muup)$&$B^2(e^{123}-4e^{124}-e^{125}-3e^{146}+$&$B=\frac{1}{4}A^3,$\\
$\muup\in\R$&$+e^{156}+e^{246}-3e^{256}-2e^{345})$&$A>\sqrt{288\muup^2+576\muup+690}$\\\cline{2-2}
&$A^2(\frac{3}{2}e^3+\frac{1}{8}e^4-\frac{3}{8}e^5+\frac{3}{2}(-\muup-1)e^6+e^7)$&\\\hline
&$-\frac{1}{\muup+1}(-e^{12}-(\muup+1)e^{34}-e^{56})$&\\\cline{2-2}
$123457I(\muup)$&$-\frac{1}{(\muup+1)^2}(-e^{135}-e^{146}+e^{236}-e^{245})$&\\\cline{2-2}
$\muup<-1$&$e^7$&\\\hline
&$-2e^{12}+4e^{34}-e^{56}$&\\\cline{2-2}
$123457I(\muup)$&$4(-2e^{135}-2e^{146}+e^{236}-e^{245})$&\\\cline{2-2}
$\muup=-1$&$e^7$&\\\hline
&$(\muup+1)^2(-e^{12}+(\muup+1)e^{34}-e^{56})$&\\\cline{2-2}
$123457I(\muup)$&$(\muup+1)^2(e^{135}+e^{146}-e^{236}+e^{245})$&\\\cline{2-2}
$\muup>-1$&$e^7$&\\\hline

&$AB(e^{12}-e^{24}+2e^{36}+e^{45}+2e^{56})$&\\\cline{2-2}
$12457N_2(\muup)$&$B^2(\frac{1}{2}e^{123}+\frac{1}{2}e^{125}-\frac{3}{2}e^{135}+$&$B=2A^3,$\\
$\muup\geq 0$&$+e^{146}-2e^{256}+e^{345})$&$A>\frac{\sqrt{2\muup^2+8\muup+59}}{4}$\\\cline{2-2}
&$A^2(-\frac{3}{2}e^3-e^4-(\muup+2)e^6+e^7)$&\\
\end{longtable}
}

\bibliographystyle{plain}
\bibliography{bibl}
\end{document}